\documentclass[]{article}
\usepackage{amsmath}
\usepackage{amssymb}
\usepackage{mathtools}
\usepackage{graphicx}
\usepackage{subfigure}
\usepackage{float}
\usepackage{bbm}
\usepackage{amsthm}
\usepackage{hyperref}
\usepackage{tikz}
\usepackage{tikz-3dplot}
\usepackage{multirow}

\newtheorem{thm}{Theorem}  
  
\newtheorem{lem}{Lemma}  
\newtheorem{prop}{Proposition} 
\newtheorem{example}{Example}
\newtheorem{defi}{Definition}
\newtheorem{rmk}{Remark}  
\newtheorem{question}{Question}

\numberwithin{table}{section}

\title{Collapsing of K3 Surfaces and Special Kähler Structures}
\author{Zexuan Ouyang}
\date{}

\begin{document}
	
	\bibliographystyle{plain}
	
	\maketitle
	
	\begin{abstract}

		We study the structure of $\mathfrak{M}_2$, the set of half-dimensional collapsing spaces of hyperkähler metrics on K3 surfaces. We show that $\mathfrak{M}_2$ consists precisely of those underlying metric spaces of integral singular special Kähler structures (SKSs) on \( \mathbb{P}^1 \). Furthermore, we establish a bijection between integral singular SKSs on \( \mathbb{P}^1 \) and Jacobian elliptic K3 surfaces with a marked holomorphic volume form. Additionally, we compute the number of Jacobian elliptic K3 surfaces that correspond to a given metric on \( \mathbb{P}^1 \).
	\end{abstract}

	\section{Introduction}
	
	A K3 surface is a simply connected, compact complex surface with a trivial canonical bundle. Siu \cite{Siu} proved that all K3 surfaces are Kähler, and by Yau’s celebrated theorem \cite{yau}, each K3 surface admits a unique hyperkähler metric in every Kähler class. Let \( \mathfrak{M} \) denote the space of isometric classes of hyperkähler metrics on K3 surfaces with unit diameter. We are interested in its Gromov--Hausdorff compactification \( \overline{\mathfrak{M}} \).
	
	Given a Gromov--Hausdorff convergent sequence of hyperkähler metrics on K3 surfaces, there is a dichotomy depending on whether the sequence is volume collapsing or not \cite{cheeger-colding}. For a non-collapsing limit space, we know the Hausdorff dimension is 4. For a collapsing limit space, by Cheeger--Tian’s \( \epsilon \)-regularity theorem \cite{cheegertian}, the $L_2$ norm of curvature concentrates near finitely many points \( \{q_1, \dots, q_n\} \) of the limit space \( X_\infty \), and by Naber--Tian \cite{naber-tian-orbifold}, there is a smooth Riemannian orbifold structure on \( X_\infty \backslash \{q_1, \dots, q_n\} \). Consequently, the Hausdorff dimension of the limit space is always an integer. This allows us to stratify the boundary of \( \mathfrak{M} \) as  
	\[
	\overline{\mathfrak{M}} \backslash \mathfrak{M} = \bigcup_{k=1}^4 \mathfrak{M}_k,
	\]
	where each \( \mathfrak{M}_k \) corresponds to limit spaces of dimension \( k \).    
	
	The moduli space of Einstein metrics on K3 surfaces, possibly with orbifold singularities, is known to be: (see \cite{kob-Tod})
	\[
	\mathcal{D} \coloneq \Gamma \backslash SO^+(3,19)/(SO(3) \times SO(19)),
	\]
	where \( \Gamma \) is the automorphism group of the K3 lattice. Anderson \cite{anderson} proved that $\mathcal{D}$ captures all non-collapsing limit spaces of hyperkähler K3 surfaces, and there is a continuous bijection from \( \mathcal{D} \) to \( \mathfrak{M} \cup \mathfrak{M}_4 \) (see also \cite[Chapter 6]{odaka-oshima}). As a special case of hyperkähler K3 orbifolds, all flat \( T^4 / \mathbb{Z}_2 \) belong to \( \mathfrak{M}_4 \). By considering Gromov--Hausdorff limits of sequences of flat \( T^4 / \mathbb{Z}_2 \), we deduce that all flat \( T^k / \mathbb{Z}_2 \) belong to \( \mathfrak{M}_k \) for \( k = 1, 2, 3 \).

	Sun--Zhang \cite{sunzhang} proved that these flat \( T^k / \mathbb{Z}_2 \) exhaust all collapsing limit spaces in \( \mathfrak{M}_1 \) and \( \mathfrak{M}_3 \). In contrast, for \( (X_\infty, g) \in \mathfrak{M}_2 \), they showed that \( (X_\infty, g) \) is homeomorphic to \( \mathbb{P}^1 \) and admits a singular special Kähler structure (SKS), and they conjectured that these SKSs admit integral monodromy \cite[Conjecture 7.4]{sunzhang}. We also remark that Odaka--Oshima \cite{odaka-oshima} proposed a interesting conjecture relating \( \overline{\mathfrak{M}} \) to the Satake compactification of \( \mathcal{D} \).
	
	In this paper, we focus on the structure of the half-dimensional collapsing space \( \mathfrak{M}_2 \), which is closely related to the Strominger--Yau--Zaslow conjecture \cite{syz}. In Section \ref{section-integrality}, we prove the following theorem, confirming Sun--Zhang's conjecture:
	
	\begin{thm} \label{thm-integrality}
		Let \( (X_\infty, g_\infty) \) be a 2-dimensional Gromov--Hausdorff limit of hyperkähler metrics on K3 surfaces. Then \( (X_\infty, g_\infty) \) admits an integral singular SKS.
	\end{thm}
	
	Given an integral SKS on a Riemann surface, one can associate a holomorphic torus fibration, also known as an algebraic integrable system \cite{freed}. In Section \ref{relation2}, we generalize this correspondence to include singular cases and prove the following theorem. Recall that a Jacobian elliptic K3 surface is an elliptic fibration of a K3 surface, \( f\colon X \to \mathbb{P}^1 \), equipped with a global section.
	
	\begin{thm} \label{correspondence}
		There is a one-to-one correspondence between:
		\begin{enumerate}
			\item Integral singular SKS on \( \mathbb{P}^1 \).
			\item Jacobian elliptic K3 surfaces \( f\colon X \to \mathbb{P}^1 \) together with a holomorphic volume form \( \Omega \) on \( X \).
		\end{enumerate}
	\end{thm}
	
	Furthermore, we define the notion of "twist transformations" for SKS (see Section \ref{section-twist}), which preserve the metric and correspond to rotating \( \Omega \) by \( e^{i\theta} \) in Theorem~\ref{correspondence}. After normalizing to unit diameter, this leads to a one-to-one correspondence between:
	
	\begin{enumerate}
		\item Integral singular SKS on \( \mathbb{P}^1 \) with unit diameter, considered up to the equivalence relation given by twist. 
		\item Jacobian elliptic K3 surfaces \( f\colon X \to \mathbb{P}^1 \).
	\end{enumerate}

	Thus, given a Jacobian elliptic K3 surface, we can associate with it a metric space \( (\mathbb{P}^1, g) \) induced by the integral singular SKS on $\mathbb{P}^1$. This metric turns out to be a generalized Kähler--Einstein metric \cite{songtian}. By Chen--Viaclovsky--Zhang’s construction \cite[Theorem~1.1]{Chen2019CollapsingRM}, these metric spaces \( (\mathbb{P}^1, g) \) are the Gromov--Hausdorff limits of collapsing sequence of hyperkähler metrics on K3 surfaces. 
	
	Combining Theorem~\ref{thm-integrality}, Theorem~\ref{correspondence}, and the results of Chen et al., we conclude that:

	\begin{thm}
		 \( \mathfrak{M}_2 \) consists precisely of the underlying metric spaces of integral singular SKSs on \( \mathbb{P}^1 \). Moreover, all the limit spaces in \( \mathfrak{M}_2 \) are given by Jacobian elliptic K3 surfaces as in Theorem~\ref{correspondence}.
	\end{thm}
	
	 This provides a complete description of \( \mathfrak{M}_2 \). Together with previous work, we have a complete classification of the Gromov--Hausdorff limit spaces of hyperkähler metrics on K3 surfaces. Furthermore, it is natural to ask the following:
	
	\begin{question}\label{question-uniqueness}
		Can we determine the Jacobian K3 surface from the singular metric on \( \mathbb{P}^1 \) according to Theorem~\ref{correspondence}?
	\end{question}
	
	As we will see, it is convenient to include information about the complex structure.  We define \( \mathfrak{M}_2^c \coloneq \{ (\mathbb{P}^1,g) \}/\sim \), where \( g \) is the underlying metric of a singular integral SKS on \( \mathbb{P}^1 \), and \( (\mathbb{P}^1,g) \sim (\mathbb{P}^1,g') \) if and only if there is an isometry from \( (\mathbb{P}^1,g) \) to \( (\mathbb{P}^1,g') \) that preserves the complex structure of \( \mathbb{P}^1 \). The forgetting map \( \mathfrak{M}_2^c \to \mathfrak{M}_2 \) is generally 2-to-1 by definition. Let \( \mathcal{M}_{\text{Jac}} \) denote the moduli space of Jacobian elliptic K3 surfaces. There is a surjective map given by Theorem~\ref{correspondence}:  
	\[
	\mathcal{F}\colon \mathcal{M}_{\text{Jac}} \to \mathfrak{M}_2^c,
	\]
	which sends a Jacobian elliptic K3 surface to the underlying metric space of the singular integral SKS on \( \mathbb{P}^1 \).

	As we will see in Section \ref{section-twist}, two SKSs on a Riemann surface are isometric if and only if they are related by a twist transformation. Thus, Question~\ref{question-uniqueness} is closely related to the uniqueness of the integral structure on a singular SKS. In Section \ref{injectivity}, we prove the following theorem, which gives a satisfactory answer to Question~\ref{question-uniqueness}:
	
	\begin{thm}\label{injective-theorem}
		Given a Jacobian elliptic K3 surface \( f\colon X \to \mathbb{P}^1 \), assume that \( \mathcal{F} \) maps \( (X, f) \) to \( (\mathbb{P}^1, g) \). Then:
		
		\begin{enumerate}
			\item If \( (X, f) \) is isotrivial, then \( (\mathbb{P}^1, g) \) has a flat metric on its regular part. If \( (\mathbb{P}^1, g) \) is a flat \( T^2 / \mathbb{Z}_2 \), then \( X \) is a Kummer surface, and \( \mathcal{F}^{-1}\left((\mathbb{P}^1, g)\right) \) is parameterized by \( \mathbb{C} \). If \( (\mathbb{P}^1, g) \) is not flat \( T^2 / \mathbb{Z}_2 \), then \( \left|\mathcal{F}^{-1}\left((\mathbb{P}^1, g)\right)\right| = 1 \).
			
			\item If \( (X, f) \) is non-isotrivial, then \( (\mathbb{P}^1, g) \) is not flat on its regular part. Take an integral singular SKS on \( \mathbb{P}^1 \) with \( (\mathbb{P}^1, g) \) as the underlying metric, and denote the monodromy group by \( G \). Then, \( \left|\mathcal{F}^{-1}\left((\mathbb{P}^1, g)\right)\right| = N(G) \), a finite number determined by \( G \). Furthermore, \( \left|\mathcal{F}^{-1}\left((\mathbb{P}^1, g)\right)\right| = 1 \) if \( (X, f) \) has 24 \( I_1 \)-fibers.
		\end{enumerate}
		
	\end{thm}
	In particular, \( \mathcal{F} \) induces a stratification of the moduli space of Jacobian elliptic K3 surfaces, and it may be an interesting problem to analyze the structure of these strata.
	
	We note that Hashimoto-Ueda \cite{reconstruct} also proved the injectivity of \( \mathcal{F} \) for a general choice of Jacobian elliptic K3 surfaces, using a totally different method.
	
	\

	{\bf Acknowledgements}: The author is indebted to his advisor, Gang Tian, for many insightful discussions and continuous encouragement. He also thank Yibing Chen, Yuxuan Chen, Qing Lan and Zhiyuan Li for helpful conversations, and Yuji Odaka, Mattias Jonsson for kind suggestion on references. The author is supported by National Key R$\&$D Program of China 2023YFA1009900.

	\section{Special Kähler Structures (SKS)}
	We will use the abbreviation SKS throughout this article. In this section, we review some basic results about SKS. The following definition is from Freed \cite{freed}, and we will see later that SKS can be determined by special pairs.

	\begin{defi} [\cite{freed}]
		An SKS on a Riemann surface \( M \) is a Kähler structure \((M, \omega, J)\) together with a torsion-free flat symplectic connection \( \nabla \) satisfying 
		\[
		(\nabla_X J)(Y) = (\nabla_Y J)(X)
		\]
		for any vector fields \( X, Y \).
	\end{defi}
	
	\subsection{Special Pairs}
	By the definition of SKS, locally there exist affine flat Darboux coordinates \( x^1, x^2 \), i.e., \( \nabla(dx^1) = \nabla(dx^2) = 0 \) and \( \omega = dx^1 \wedge dx^2 \).
	
	\begin{lem} [\cite{freed}]  \label{lemma-special-pair}
		Given affine flat Darboux coordinates \( (x^1, x^2) \) as above, we can uniquely determine two holomorphic 1-forms \( (dz, dw) \) such that \( \mathrm{Re}(dz) = dx^1 \) and \( -\mathrm{Re}(dw) = dx^2 \).
	\end{lem}

	\begin{proof}
		Note that \( \text{Re} \) gives a bijection from \( T^*_{1,0} M \) to \( T^* M \), and \( dx^1 + \mathrm{i}J(dx^1) \) maps to \( dx^1 \).
		
		We write the complex structure as \( J = J^l_k \partial_{x^l} \otimes dx^k \). By \( (\nabla_X J)(Y) = (\nabla_Y J)(X) \) we have
		\[
		d(J(dx^1)) = d(J^1_k dx^k) = \partial_l J^1_k dx^l \wedge dx^k = 0.
		\]
		Thus \( d(dx^1 + \mathrm{i}J(dx^1))=0 \), and there exist a holomorphic function \( z \), such that \( dz = dx^1 + \mathrm{i}J(dx^1) \), and \( dz \) is uniquely determined by \( dx^1 \). The calculation for \( dw \) is similar.
	\end{proof}
	Denote \( \tau = \frac{dw}{dz} \), then the Kähler form can be written as:
	\[
	\omega =\text{Re}(dw)\wedge \text{Re}(dz)= \frac{1}{2} \text{Re}(dw \wedge d\bar{z}) = \frac{\mathrm{i}}{2} \text{Im}(\tau) \, dz \wedge d\bar{z}.
	\]
	In particular, we have \( \text{Im}(\tau) > 0 \).
	
	Given two holomorphic 1-forms \( (dz, dw) \) with \( \text{Im}(\tau) > 0 \), we can determine \( \nabla \) by \( \nabla ( \text{Re}(dz)) = \nabla ( \text{Re}(dw)) = 0 \). The calculation in Lemma~\ref{lemma-special-pair} can be reversed to derive \( (\nabla_X J)(Y) = (\nabla_Y J)(X) \). Thus, we can determine an SKS by \( (dz, dw) \):

	\begin{lem}\label{holomorphic-coordinate-representation}
		Given two holomorphic functions \( (dz, dw) \) on an open set \( D \subset \mathbb{C} \) with \( \mathrm{Im}(\tau) > 0 \), then there is an SKS on \( D \), with \( (\mathrm{Re}(z), -\mathrm{Re}(w)) \) as affine flat Darboux coordinates.
	\end{lem}
	
	\begin{defi}
		Given an SKS on a Riemann surface \( M \), we call a locally defined pair of holomorphic 1-forms \( (dz, dw) \) a \textit{special pair} if \( (\mathrm{Re}(z), -\mathrm{Re}(w)) \) are affine flat Darboux coordinates of the SKS.
	\end{defi}
	
	\subsection{Monodromy Representations}
	
	Given a Riemann surface \( (M, p) \) with an SKS, let \( (dz, dw) \) be a locally defined special pair near $p$. We can extend the special pair \( (dz, dw) \) along any path, so that \( (dz, dw) \) actually gives a special pair in the universal cover of \( M \). It is also convenient to view \( (dz, dw) \) as a multi-valued pair on \( M \). We observe that different special pairs of a given SKS are related by the group \( \mathrm{SL}_2(\mathbb{R}) \).
	
	Consider a loop \( \gamma \) based at \( p \). We can extend the special pair \( (dz, dw) \) along \( \gamma \), resulting in a new special pair \( (d\tilde{z}, d\tilde{w}) \) in a neighborhood of \( p \), we have
	\[
	(d\tilde{z}, d\tilde{w}) = (dz, dw)\begin{pmatrix}
		a & b \\
		c & d
	\end{pmatrix}
	\]
	for some matrix \( \begin{pmatrix} a & b \\ c & d \end{pmatrix} \in \mathrm{SL}_2(\mathbb{R}) \).
	
	It is straightforward to verify that this defines a group homomorphism \( \rho\colon \pi_1(M, p) \to \mathrm{SL}_2(\mathbb{R}) \), such that
	\[
	\gamma \mapsto A_\gamma = \begin{pmatrix} a & b \\ c & d \end{pmatrix}.
	\]
	
	\begin{defi}
		We define the monodromy representation of \( M \) at \( p \) to be the group homomorphism \( \rho \) given above.
	\end{defi}
	
	A representation depends on the choice of special pair \( (dz, dw) \) at \( p \). We say a monodromy representation is integral if its image lies in \( \mathrm{SL}_2(\mathbb{Z}) \), and we say \( (dz, dw) \) is an integral special pair if so. In the integral case, \( \langle dz, dw \rangle \) gives a well-defined lattice on \( T^*_{1,0} M \).

	\begin{defi}
		An integral SKS on a Riemann surface is defined to be a tuple \((M, \omega, J,\nabla,\Lambda) \), where \((M, \omega, J,\nabla) \) gives an SKS on $M$, and \( \Lambda = \langle dz, dw \rangle \) is a lattice on \( T^*_{1,0} M \) generated by an integral special pair \( (dz, dw) \).
	\end{defi}
	By definition, \( (dz, dw) \) and \( (dz, dw)P \) give the same integral SKS on $M$ for $P\in \mathrm{SL}_2(\mathbb{Z}) $. However, an SKS on \( M \) can admit several different integral structures. See Section \ref{injectivity}.
	
	\subsection{Twist Transformations}\label{section-twist}
	
	Given an SKS on a Riemann surface with a special pair \( (dz, dw) \) at some point \( p \), consider the holomorphic 1-forms \( (e^{\mathrm{i}\theta} dz, e^{\mathrm{i}\theta} dw) \), where \( \theta \in \mathbb{R} \). By Lemma~\ref{holomorphic-coordinate-representation}, locally \( (e^{\mathrm{i}\theta} dz, e^{\mathrm{i}\theta} dw) \) defines a new SKS with the same metric. 
	
	It is straightforward to verify that this construction is independent of the choice of special pair, yielding a well-defined SKS on \( M \). Moreover, the monodromy of the new SKS, in the special pair \( (e^{\mathrm{i}\theta} dz, e^{\mathrm{i}\theta} dw) \), is identical to that of the original SKS in the special pair \( (dz, dw) \). We call this construction "twist by \( \theta \)".
	
	We now turn our attention to the metric aspects of SKS. Since we are working on Riemann surfaces, both the complex structure and Kähler structure can be determined from the Riemannian metric. The following two propositions were proved in \cite[Proposition~ 34]{haydyshyper}, using more refined structures related to SKS. We provide a direct proof here.
	
	\begin{prop}\label{sks-with-same-metric}
		Two SKS on a Riemann surface are isometric if and only if they are related by a twist.
	\end{prop}
	
	\begin{proof}
		We have already shown that a twist does not affect the metric. Now we prove the converse.
		
		Let \( (dz_1, dw_1) \) and \( (dz_2, dw_2) \) be two special pairs defining isometric metrics, then we have
		\[
		\omega = \frac{\mathrm{i}}{2} \text{Im}(\tau_1) \, dz_1 \wedge d\bar{z}_1 = \frac{\mathrm{i}}{2} \text{Im}(\tau_2) \, dz_2 \wedge d\bar{z}_2.
		\]
		
		\

		\textbf{Claim:} After applying a \( \mathrm{SL}_2(\mathbb{R}) \) transformation to \( (dz_1, dw_1) \), we can assume \( dz_2 = e^{\mathrm{i}\theta} dz_1 \) for some \( \theta \in \mathbb{R} \).
		
		Let \( h = \frac{dz_2}{dz_1} \). If \( \partial_{z_1} h \equiv 0 \), then \( h \) is constant, and the claim follows. If \( \partial_{z_1} h \not\equiv 0 \), a calculation shows:
		\[
		0 = \partial_{\bar{z}_1} \partial_{z_1} (\text{Im}(\tau_2)) = \partial_{\bar{z}_1} \partial_{z_1} \left( \frac{\text{Im}(\tau_1)}{|h|^2} \right) = \frac{|\partial_{z_1} h|^2}{|h|^4} \cdot \text{Im}\left( \tau_1 - \partial_{z_1} \tau_1 \frac{h}{\partial_{z_1} h} \right),
		\]
		which implies that \( \tau_1 - \frac{h}{\partial_{z_1} h} \partial_{z_1} \tau_1 = r \) for some constant \( r \in \mathbb{R} \). This leads to:
		\[
		\partial_{z_1} \left( \ln\left( \frac{\tau_1 - r}{h} \right) \right) = 0,
		\]
		so we obtain \( h = c (\tau_1 - r) \) for some constant \( c \in \mathbb{C} \), and \( dz_2 = c \left( dw_1 - r dz_1 \right) \). Thus, by applying an \( \mathrm{SL}_2(\mathbb{R}) \) transformation, we have proved the claim.
		
		Now we can assume \( dz_2 = e^{\mathrm{i}\theta} dz_1 \). From the condition \( \frac{\mathrm{i}}{2} \text{Im}(\tau_1) \, dz_1 \wedge d\bar{z}_1 = \frac{\mathrm{i}}{2} \text{Im}(\tau_2) \, dz_2 \wedge d\bar{z}_2 \), we deduce \( \text{Im}(\tau_1) = \text{Im}(\tau_2) \). Thus, \( dw_2 = e^{\mathrm{i}\theta} (dw_1 + r' dz_1) \) for some \( r' \in \mathbb{R} \), and by applying another \( \mathrm{SL}_2(\mathbb{R}) \) transformation to \( (dz_1, dw_1) \), we may assume that \( (dz_2, dw_2) = e^{\mathrm{i}\theta} (dz_1, dw_1) \). This shows that the two SKS are related by a twist by \( \theta \).
	\end{proof}
	
	\begin{prop} \label{twist-give-same-sks}
		Given an SKS on a Riemann surface \( (M, J, \omega, \nabla) \), twisting by \( \theta \) gives the same SKS as the original if and only if \( \theta \in \pi \mathbb{Z} \) or the metric is flat.
	\end{prop}
	
	\begin{proof}
		If \( \theta = k \pi \) for some integer \( k \), the twist does not affect the SKS. If the metric is flat, then by Proposition~\ref{sks-with-same-metric}, locally the SKS is related by a twist to the standard SKS on \( \mathbb{R}^2 \), given by the special pair \( (z, \mathrm{i}z) \). It follows that the standard SKS is invariant under twists.
		
		Conversely, if the twist by \( \theta \notin \pi \mathbb{Z} \) gives the same SKS as the original, then we have \( e^{\mathrm{i}\theta} dz = r_1 dz + r_2 dw \) for some real numbers \( r_1, r_2 \) with \( r_2 \neq 0 \). Therefore, \( dw = c dz \) for some constant \( c \in \mathbb{C} \), and the Kähler form given by \( \frac{1}{2} \text{Re}(dw \wedge d\bar{z}) \) leads to a flat metric.
	\end{proof}
	
	\begin{rmk}
		Twisting by $\frac{\pi}{2}$ is equivalent to the notion of Legendre transform (e.g. \cite{loftin}), which plays an important role to mirror symmetry theory.  
	\end{rmk}
	
	\subsection{Relation with Elliptic Fibrations: Regular Part}
	
	A classical result shows that integral SKS naturally arise in the base space of an algebraic integrable system \cite{freed}, and this procedure can be reversed. In mathematical terms, this leads to the following proposition. See also \cite{Chen2019CollapsingRM,markgross97,GW}.
	
	\begin{prop}\label{bijection-regular-region}  
		Let \( B \) be a Riemann surface. There is a bijection between the following sets:
		\begin{enumerate}
			\item Integral SKS on \( B \).
			\item Holomorphic torus fibration \( f\colon X \to B \) with a zero-section together with a holomorphic volume form \( \Omega \) on \( X \).
		\end{enumerate}
		Moreover, twisting by \( \theta \) of SKS on \( B \) corresponds to multiplying \( e^{\mathrm{i} \theta} \) to \( \Omega \).
	\end{prop}
	\begin{proof} Let \( f\colon X \to B \) be a holomorphic torus fibration over a Riemann surface with a zero-section and a holomorphic volume form \( \Omega \). From \cite[sections 2 and 7]{markgross97}, there is uniquely a holomorphic map \( h\colon T^*B \to X \) that preserves the zero-section such that \( h^* \Omega = \Omega_{\text{can}} \), where \( \Omega_{\text{can}} \) is the canonical holomorphic volume form on \( T^*B \). Thus, we can write \( X = T^*B / \Lambda \), where \( \Lambda \) is a holomorphic lattice generated by two multi-valued holomorphic 1-forms \( dz, dw \). We can assume that \( \text{Im}\left( \frac{dw}{dz} \right) > 0 \), so \( (dz, dw) \) defines an integral SKS on \( B \).
		
		Conversely, given an integral SKS on \( B \) with \( \Lambda = \langle dz, dw \rangle \). We can then recover the space \( X = T^*B / \Lambda \). The canonical holomorphic volume form $\Omega$ on $T^*B$ is invariant under translation of $adz+bdw$, thus gives a canonical holomorphic volume form on $X$.      
	\end{proof}
	
	Recall that the Kähler form of the SKS on \( B \) is $\omega = \frac{\mathrm{i}}{2} \text{Im}(\tau) \, dz \wedge d\bar{z},$ where \( \tau = \frac{dw}{dz} \). The curvature is calculated to be 
	\[
	\text{Ric}(\omega) = - \partial \bar{\partial} \log (\text{Im}(\tau)) = \omega_{\text{WP}},
	\]
	where \( \omega_{\text{WP}} \) is the induced Weil-Petersson metric on \( B \). Thus, \( \omega \) is a generalized Kähler--Einstein metric \cite{songtian} with respect to \( f: X \to B \).

	\section{Structures of 2-Dimensional Collapsing Limit Spaces}\label{structure-of-limit-space}
	
	\subsection{Singularity Structures}
	
	In \cite{sunzhang}, Sun--Zhang classified all local singular SKS models that can appear as Gromov--Hausdorff limits of hyperkähler metrics. We adopt their definition of singular SKS as follows. For other types of singularities, see \cite{haydyslocal,hay15,haydyshyper}.
	
	Denote
	\[
	R_\theta \coloneq \begin{pmatrix*}[r]
		\cos( \theta) & -\sin( \theta) \\
		\sin(\theta) & \cos( \theta)
	\end{pmatrix*}, \quad
	I_r \coloneq \begin{pmatrix}
		1 & r \\
		0 & 1
	\end{pmatrix}.
	\]
	
	\begin{defi}  [Definition~3.19 of \cite{sunzhang} ] \label{definition-singularity}
		A singular SKS on a Riemann surface \( M \) is an SKS on \( M \setminus \{ p_1, \dots, p_n \} \), such that for each \( p_k \), we can take a small disk \( \Delta \) containing \( p_k \) and a holomorphic coordinate \( \zeta \) on \( \Delta \) with \( \zeta(p_k) = 0 \), and the SKS on \( \Delta^* = \Delta \setminus \{ 0 \} \) is described in Table~\ref{table-singular-model}.

		\begin{table}[h]
			\centering
			\caption{Singular Models}
			\label{table-singular-model}
			\renewcommand\arraystretch{1.2}
			\begin{tabular}{|c|c|l|c|}
				\hline
				\textbf{Type} & \textbf{Kähler Form} & \qquad \quad \textbf{Special Pair}& \textbf{Monodromy} \\
				\hline
				(1) & 
				\( - \frac{\mathrm{i} }{2} \left( \log |\zeta| + \text{Im}(f) \right) d\zeta \wedge d\bar{\zeta} \) & 
				\( \begin{array}{l}
					dz = d\zeta \\
					dw = -\left( \mathrm{i} \log \zeta + f \right) d\zeta \\
					\tau = -\mathrm{i} (\log \zeta + f)
				\end{array} \) & 
				\( I_{2\pi} \) \\
				\hline
				(2) & 
				\( -\frac{\mathrm{i} }{2|\zeta|} \left( \log |\zeta| + \text{Im}(f) \right) d\zeta \wedge d\bar{\zeta} \) & 
				\( \begin{array}{l}
					dz = \zeta^{-\frac{1}{2}} d\zeta \\
					dw = -\left( \mathrm{i} \log \zeta + f \right) \zeta^{-\frac{1}{2}} d\zeta \\
					\tau = -(\mathrm{i} \log \zeta + f)
				\end{array} \) & 
				\( -I_{2\pi} \) \\
				\hline
				(3) & 
				\( \frac{\mathrm{i} }{2} \left( 1 - |f|^2 \right) |\zeta|^{2\beta-2} d\zeta \wedge d\bar{\zeta} \) & 
				\( \begin{array}{l}
					dz = (1 - f) \zeta^{\beta-1} d\zeta \\
					dw = \mathrm{i} (1 + f) \zeta^{\beta-1} d\zeta \\
					\tau = \mathrm{i} \frac{1 + f}{1 - f}
				\end{array} \) & 
				\( R_{2\pi\beta} \) \\
				\hline
			\end{tabular}
		\end{table}
		
		In cases (1) and (2), \( f \) is a holomorphic function on \( \Delta \).
		
		In case (3), we require \( \beta \in \left\{ \frac{1}{2}, \frac{1}{3}, \frac{2}{3}, \frac{1}{4}, \frac{3}{4}, \frac{1}{6}, \frac{5}{6} \right\} \), and \( f \) is a multivalued holomorphic function on \( \Delta \) of  the following form (where $g$ is a holomorphic function on $\Delta$): 
		
		\begin{itemize}
			\item If \( \beta = \frac{1}{2} \), then \( f \) is a holomorphic function on \( \Delta \) with \( |f| < 1 \).
			\item If \( \beta = \frac{1}{4}, \frac{3}{4} \), then \( f=z^\frac{1}{2}g\).
			\item If \( \beta = \frac{1}{3}, \frac{5}{6} \), then \( f=z^\frac{1}{3}g\).
			\item If \( \beta =  \frac{2}{3}, \frac{1}{6} \), then \( f=z^\frac{2}{3}g\).
		\end{itemize}

		In the third column, the monodromy refers to the monodromy along a counterclockwise small circle centered at \( p_k \).
	\end{defi}

	\begin{rmk}
		Take $\beta=\frac{1}{3}$, for example, \cite[Definition~3.19]{sunzhang} only assumes that $f=F^\frac{1}{3}$ for some holomorphic $F$ with $F(0)=0$, but since the monodromy is a linear transformation, we must have $f=z^\frac{1}{3}g$ as above. The same holds for other values of  $\beta$.
	\end{rmk}
	
	\begin{rmk} \label{property-singularity}
		The singularities described in Definition~\ref{definition-singularity} satisfy the following properties:
		
		\begin{enumerate}
			\item The diameter of these singular models is finite.
			\item The monodromy of these singularities is conjugate to an integral matrix.
			\item The tangent cone at these singularities is unique and is a flat cone with cone angle \( \theta \in (0, 2\pi] \).
		\end{enumerate}
	\end{rmk}
	
	\begin{rmk}\label{positivity-singularity}
		From Table~\ref{table-singular-model}, we observe that the monodromy of a singularity is never \( \mathrm{SL}_2(\mathbb{R}) \) conjugated to \( \pm I_{-1} \). See, e.g., \cite[Lemmas 3.17 and 3.18]{sunzhang} for more details on the conjugacy classes of \( \mathrm{SL}_2(\mathbb{R}) \) and \( \mathrm{SL}_2( \mathbb{Z}) \).
	\end{rmk}
	
	\begin{rmk}
		
		Given a singular SKS on a compact Riemann surface \((M,\omega) \), we known that $\text{Ric}(\omega) = \omega_{\text{WP}}\ge0,$ and the cone angle at singularities is not greater than $2\pi$. Thus, by the Gauss-Bonnet formula, \( M \) is either homeomorphic to \( \mathbb{P}^1 \) or a flat torus.  ($cf.$ \cite[Theorem~2]{Lu})
		
	\end{rmk}

	\subsection{Integrality of SKS}\label{section-integrality}
	
	In \cite{sunzhang}, Sun--Zhang proved the following:
	
	\begin{thm} \label{thm-sun-zhang-limit-space}
		Let \((X_\infty, g_\infty)\) be a 2-dimensional Gromov--Hausdorff limit space of hyperkähler metrics of K3 surfaces. Then \((X_\infty, g_\infty)\) is isometric to a singular SKS on \( \mathbb{P}^1 \).
	\end{thm}
	
	In this section, we will prove the singular SKS on \( \mathbb{P}^1 \) has integral monodromy in a special pair $(dz,dw)$, thereby confirming Conjecture 7.4 of \cite{sunzhang}.

	We first recall some of the arguments in \cite{sunzhang}. Let \((X_\infty, g_\infty)\) be as in Theorem~\ref{thm-sun-zhang-limit-space}, and let \( R \) denote the regular part of \( X_\infty \). Furthermore, let \( Q \subset\subset R \) be a compact connected domain with a smooth boundary. By the results of Cheeger-Fukaya-Gromov \cite{fucayacheegergromov}, for sufficiently large $j$, there exists a fiber bundle map \( F_j\colon Q_j \to Q \) for some \( Q_j \subset X_j^4 \). The hyperkähler metric on \( Q_j \) can be approximated by an \( \mathcal{N} \)-invariant hyperkähler metric, and these metrics define a sequence of integral SKS structures on \( Q \), converging smoothly to the limit metric on \( Q \). In summary, we have:
	\begin{prop} [{\cite[Theorem~3.27]{sunzhang}}] \label{converge}
		On any connected compact domain \( Q \subset\subset R \) with smooth boundary, there exists a sequence of integral SKS on \( Q \) that converges to \( g_\infty \) in the \( C^\infty \) Cheeger-Gromov sense.
	\end{prop}
	
	Let \((dz_i, dw_i)\) be a sequence of integral special pairs given by Proposition~\ref{converge}, and let \( \tau_i = \frac{d w_i}{d z_i} \). Choose a base point \( q \in Q \). By applying a transformation in \( \mathrm{SL}_2(\mathbb{Z}) \), we can select \((dz_i, dw_i)\) such that \( \tau_i(q) \) lies within the fundamental domain of \( \mathrm{SL}_2(\mathbb{Z}) \backslash \mathcal{H} \), which is known to be:
	\[
	\mathcal{D} = \left\{ z \in \mathbb{H} \mid -\frac{1}{2} \le \text{Re}(z) \le \frac{1}{2}, \quad |z| \ge 1 \right\}.
	\]
	
	Let \( \lambda_i = \sqrt{\text{Im}(\tau_i(q))} \). We can then rescale \( (dz_i, dw_i) \) by:
	\[
	(dz'_i, dw'_i)=(dz_i, dw_i) P_i,
	\]
	where 
	\[
	P_i = \begin{pmatrix} \lambda_i & 0 \\ 0 & \lambda_i^{-1} \end{pmatrix}.
	\]
	Therefore \( \tau_i' = \frac{dw_i'}{dz_i'} \) satisfies \( \text{Im}(\tau_i'(q)) = 1 \). However, the monodromy in the special pair \( (dz_i', dw_i') \) may not be integral.
	
	By the estimate of harmonic functions, after possibly taking a subsequence, we can assume that \( (dz_i', dw_i') \) converge smoothly to \( (dz_\infty', dw_\infty') \) in the universal cover of \( Q \) (see \cite[Section 3.2.2]{sunzhang} for details). As a result, \( (dz_\infty', dw_\infty') \) defines an SKS on \( Q \) with the metric \( g_\infty \).
	
	Let \( \gamma \) be a loop on \( Q \) based at $q$. In the following, we will write $A_{\gamma,i},A'_{\gamma,i},A'_{\gamma,\infty}$ for the monodromy along \( \gamma \) in the special pairs \( (dz_i, dw_i), (dz'_i, dw'_i), (dz'_\infty, dw'_\infty) \), respectively. We have $A'_{\gamma, i} = P_i^{-1} A_{\gamma, i} P_i$ and \( A'_{\gamma, i} \to A'_{\gamma, \infty} \).

	\begin{proof}[Proof of Theorem~\ref{thm-integrality}]
		Let \( Q = \mathbb{P}^1 \backslash \left( \bigcup_{k=1}^n B(p_k) \right) \), where \( B(p_k) \) are disjoint disks centered at \( p_k \). 
		The proof proceeds in two cases: whether \( \lambda_i^2 = \text{Im}(\tau_i) \) is bounded or not.
		
		The bounded case is easier and already proved in \cite[Remark~3.18]{sunzhang}. In this case, we can assume that \( (dz_i, dw_i) \) converge smoothly to \( (dz_\infty, dw_\infty) \) by taking a subsequence, thus we have the smooth convergence of monodromy and \( (dz_\infty, dw_\infty) \) defines an integral SKS on \( (Q, g_\infty) \).

		In the unbounded case, we may assume \( \lambda_i \to \infty \) by taking a subsequence.
		For any loop \( \gamma \) based at \( q \), we write
		\[
		A_{\gamma, i} = \begin{pmatrix} a_i & b_i \\ c_i & d_i \end{pmatrix} \in \mathrm{SL}_2(\mathbb{Z}), \quad A'_{\gamma, \infty} = \begin{pmatrix} a & b \\ c & d \end{pmatrix} \in \mathrm{SL}_2( \mathbb{R}).
		\]
		As in \cite[Remark~3.18]{sunzhang}, we have
		\[
		A'_{\gamma, i} = P_i^{-1} A_{\gamma, i} P_i = \begin{pmatrix} a_i & \lambda_i^{-2} b_i \\ \lambda_i^2 c_i & d_i \end{pmatrix} \to \begin{pmatrix} a & b \\ c & d \end{pmatrix} = A'_{\gamma, \infty}.
		\]
		Thus, for sufficiently large \( i \), we must have \( c_i = 0 \), and \( a_i = d_i = \pm 1 \). Therefore,
		
		\[
		A'_{\gamma, \infty} = \pm \begin{pmatrix} 1 & b \\ 0 & 1 \end{pmatrix}.
		\]
		
		From the topology of $Q$, we can choose loops \( \gamma_k \) based at \( q \) for each \( p_k \), which go counterclockwise around \( p_k \), and satisfy the following:
		\[
		\gamma_1 \gamma_2 \dots \gamma_n = \text{Id} \in \pi_1(Q, q).
		\]
		We can then write
		\[
		A'_{\gamma_k, \infty} = \pm \begin{pmatrix} 1 & b_k \\ 0 & 1 \end{pmatrix}.
		\]
		Thus,
		\[
		\text{Id} = A'_{\gamma_1 \gamma_2 \dots \gamma_n, \infty} = A'_{\gamma_1, \infty} A'_{\gamma_2, \infty} \dots A'_{\gamma_n, \infty} = \pm \begin{pmatrix} 1 & \sum b_k \\ 0 & 1 \end{pmatrix}.
		\]
		
		As noted in Remark~\ref{positivity-singularity}, the monodromy around these singularities \( p_k \) is not \( \mathrm{SL}_2(\mathbb{R}) \)-conjugate to \( \pm I_{-1} \), hence \( b_k \geq 0 \). Therefore, \( b_k = 0 \) for all \( k \), and thus \( (dz_\infty', dw_\infty') \) gives an integral SKS on $(Q,g_\infty)$. Hence we have an integral singular SKS on $(X_\infty,g_\infty)$.
	\end{proof}
	
	The following example shows \( \text{Im}(\tau_i) \) can be unbounded.
	
	\begin{example}
		\label{example-kummer}
		Given \( \tau \) with \( \mathrm{Im}(\tau) > 0 \), let \( G \) be the group of \( \text{Aut}(\mathbb{C}) \) generated by \( T_1\colon z \mapsto z + 1 \), \( T_2\colon z \mapsto z + \tau \), and \( R\colon z \mapsto -z \). Then the quotient \( \mathbb{C} / G \) is a flat orbifold \( T^2 / \mathbb{Z}_2 \) with four singularities, and the metric is given by \( g=|dz|^2 \). 
		
		Note that \( (dz, \mathrm{i}dz) \) gives a special pair near \( \left( \frac{1}{4}, 0 \right) \), and extends to a special pair in the universal cover of the regular part of \( T^2 / \mathbb{Z}_2 \). The monodromy along any singularity is \( -\mathrm{Id} \) in any special pair \( (dz', dw') \). Thus, the integral structure is highly unique. We can take \( (dz_i, dw_i) = (\lambda_i^{-1} dz, \lambda_i  \mathrm{i}d z) \), with \(  \mathrm{Im}(\tau_i) = \lambda_i^2 \to \infty \).
		
	\end{example}
	
	In section \ref{section-isotrivial}, We will see that this example corresponds to collapsing of Kummer surfaces.

	\subsection{Relation with Elliptic Fibrations: Global} \label{relation2}
	
	In this section, we extend Proposition~\ref{bijection-regular-region} to a global form include singularities.  The following proposition gives the integral models of singular SKS.
	
	\begin{prop}
		\label{correspondence-singularity}
		
		Given a singular SKS model on a small disk \( \Delta \) as in Table~\ref{table-singular-model}, with  $(dz,dw)$ as an integral special pair. We can apply a $\mathrm{SL}_2(\mathbb{Z})$ transformation to $(dz,dw)$, such that in a suitable holomorphic coordinate $\zeta$, the special pair corresponds to a model in Table~\ref{table-two-singular-model}. In particular, the singular type of an integral SKS matches Kodaira’s table of singular fibers \cite{kod23}.
		
		\begin{table}[h]
			\centering
			\caption{Integral Singular Models}
			\label{table-two-singular-model}
			\renewcommand\arraystretch{1.2}
			\begin{tabular}{|l|l|l|c|}
				\hline
				\textbf{Type} & \textbf{\qquad Special Pairs}& \textbf{\(\beta\)} & \textbf{Monodromy} \\
				\hline
				\textbf{$I_n \ (n > 0)$} & 
				\( \begin{array}{l}
					dz = d\zeta \\
					dw = \frac{n}{2\pi \mathrm{i}} (\log \zeta + f) d\zeta
				\end{array} \)& $\backslash$ & \( I_n \) \\
				\hline
				\textbf{$I_n^* \ (n > 0)$} & 
				\( \begin{array}{l}
					dz = \zeta^{-\frac{1}{2}} d\zeta \\
					dw = \frac{n}{2\pi \mathrm{i}} \zeta^{-\frac{1}{2}} (\log \zeta + f) d\zeta
				\end{array} \)& $\backslash$ & \( -I_n \) \\
				\hline
				\textbf{$I_0^*$} & 
				\( \begin{array}{l}
					dz = g \zeta^{-\frac{1}{2}} d\zeta \\
					dw = h \zeta^{-\frac{1}{2}} d\zeta
				\end{array} \)
				& \( \frac{1}{2} \) & \( -\mathrm{Id} \)\\
				\hline
				\textbf{$II$} & 
				\( \begin{array}{l}
					dz = (1 - f) \zeta^{-\frac{1}{6}} d\zeta \\
					dw = e_3 (1 - e_3 f) \zeta^{-\frac{1}{6}} d\zeta
				\end{array} \)
				& \( \frac{5}{6} \) & 
				\( \begin{pmatrix} 0 & 1 \\ -1 & 1 \end{pmatrix} \) \\
				\hline
				\textbf{$II^*$} & 
				\( \begin{array}{l}
					dz = (1 - f) \zeta^{-\frac{5}{6}} d\zeta \\
					dw = e_3 (1 - e_3 f) \zeta^{-\frac{5}{6}} d\zeta
				\end{array} \)
				& \( \frac{1}{6} \) & 
				\( \begin{pmatrix} 1 & -1 \\ 1 & 0 \end{pmatrix} \) \\
				\hline
				\textbf{$III$} & 
				\( \begin{array}{l}
					dz = (1 - f) \zeta^{-\frac{1}{4}} d\zeta \\
					dw = \mathrm{i}(1 + f) \zeta^{-\frac{1}{4}} d\zeta
				\end{array} \)
				& \( \frac{3}{4} \) & 
				\( \begin{pmatrix} 0 & 1 \\ -1 & 0 \end{pmatrix} \) \\
				\hline
				\textbf{$III^*$} & 
				\( \begin{array}{l}
					dz = (1 - f) \zeta^{-\frac{3}{4}} d\zeta \\
					dw = \mathrm{i}(1 + f) \zeta^{-\frac{3}{4}} d\zeta
				\end{array} \)
				& \( \frac{1}{4} \) & 
				\( \begin{pmatrix} 0 & -1 \\ 1 & 0 \end{pmatrix} \) \\
				\hline
				\textbf{$IV$} & 
				\( \begin{array}{l}
					dz = (1 - f) \zeta^{-\frac{1}{3}} d\zeta \\
					dw = e_3 (1 - e_3 f) \zeta^{-\frac{1}{3}} d\zeta
				\end{array} \)
				& \( \frac{2}{3} \) & 
				\( \begin{pmatrix} -1 & 1 \\ -1 & 0 \end{pmatrix} \) \\
				\hline
				\textbf{$IV^*$} & 
				\( \begin{array}{l}
					dz = (1 - f) \zeta^{-\frac{2}{3}} d\zeta \\
					dw = e_3 (1 - e_3 f) \zeta^{-\frac{2}{3}} d\zeta
				\end{array} \)
				& \( \frac{1}{3} \) & 
				\( \begin{pmatrix} 0 & -1 \\ 1 & -1 \end{pmatrix} \) \\
				\hline
			\end{tabular}
		\end{table}
		
		\begin{enumerate}
			\item In the cases $I_n$ and $I^*_n$ ($n>0$), \( f \) is a holomorphic function.
			\item In the \( I_0^* \) case, \( g \) and \( h \) are holomorphic functions satisfying \( \mathrm{Im}\left(\frac{g}{h}\right) > 0 \).
			\item In other cases, \( f \) satisfies the condition in Definition~\ref{definition-singularity} according to \( \beta \), and \( e_3\coloneq e^{\frac{2\pi\mathrm{i}}{3}} \).
		\end{enumerate}
	\end{prop}

	We will need the following lemma, the proof is simple.
	\begin{lem}
		\label{commute-elliptic}
		If \( A \in \mathrm{SL}_2(\mathbb{R}) \) commutes with \( R_\alpha \) for \(\alpha \notin \pi \mathbb{Z} \), then \( A = R_\theta \) for some \( \theta \in \mathbb{R} \).
	\end{lem}
	
	\begin{proof}[Proof of Proposition~\ref{correspondence-singularity}]
		
		Consider a singular SKS model given by a special pair \( (dz_0, dw_0) \) as in Table~\ref{table-singular-model}. Let \( (dz, dw) \) be an integral special pair. We have  
		\[
		(dz, dw) = (dz_0, dw_0) P
		\]
		for some \( P \in \mathrm{SL}_2(\mathbb{R}) \).
		
		We will prove proposition \ref{correspondence-singularity} for singularities of type (1) and type (3) with \( \beta =  \frac{1}{2},\frac{5}{6}\) in Table~\ref{table-singular-model}, the other cases follow similarly.
		
		\
		
		\textbf{Type (1):}  We have  $(dz_0,dw_0)=(d\zeta ,-\left( \mathrm{i} \log \zeta + f \right) d\zeta)$. By applying a transformation in \( \mathrm{SL}_2(\mathbb{Z}) \), we can assume the monodromy in \( (dz, dw) \) is \( I_n \). This implies that \( dz \) is single-valued, and we must have \( dz = rd\zeta \) for some \( r \in \mathbb{R} \). After rescaling to \( \zeta \), we can assume \( dz = d\zeta \), and \( dw= \frac{n}{2\pi \mathrm{i}} (\log \zeta + f') \) for a holomorphic function $f'$.
		
		\
		
		\textbf{Type (3) with \( \beta = \frac{1}{2} \):}  We have $(dz_0,dw_0)=((1 - f) \zeta^{-\frac{1}{2}} d\zeta ,\mathrm{i} (1 + f) \zeta^{-\frac{1}{2}} d\zeta)$, thus $(dz, dw) = (dz_0, dw_0) P$ is of the form $(g\zeta^{-\frac{1}{2}} d\zeta ,h\zeta^{-\frac{1}{2}} d\zeta)$.
		
		\
		
		\textbf{Type (3) with \( \beta = \frac{5}{6} \):}  We have  $(dz_0,dw_0)=((1 - f) \zeta^{-\frac{1}{6}} d\zeta ,\mathrm{i} (1 + f) \zeta^{-\frac{1}{6}} d\zeta)$. By applying a transformation in \( \mathrm{SL}_2(\mathbb{Z}) \), we can assume the monodromy in \( (dz, dw) \) is
		\[
		\begin{pmatrix}
			0 & 1 \\
			-1 & 1
		\end{pmatrix} = P^{-1} R_{2\pi\beta} P.
		\]
		Let \( (dz',d w') = (dz_0, dw_0) X \), where \( X = r_0 \begin{pmatrix}
			1 & -\frac{1}{2} \\
			0 & \frac{\sqrt{3}}{2}
		\end{pmatrix} \) with \( r_0 = \sqrt{\frac{2}{\sqrt{3}}} \). Then the monodromy in \( (dz', dw') \) is also
		
		\[
		\begin{pmatrix}
			0 & 1 \\
			-1 & 1
		\end{pmatrix} = X^{-1} R_{2\pi\beta} X.
		\]
		Thus, \( P X^{-1} \) commutes with \( R_{2\pi\beta} \).
		
		By Lemma~\ref{commute-elliptic}, \( P = R_\theta X \) for some \( \theta \in \mathbb{R} \), and thus
		\[
		(dz, dw) = (dz_0, dw_0) P = r_0 e^{\mathrm{i} \theta} \left( (1 - e^{-2\mathrm{i} \theta} f) \zeta^{-\frac{1}{6}} d\zeta, \quad e_3 (1 - e_3 e^{-2\mathrm{i} \theta} f) \zeta^{-\frac{1}{6}} d\zeta \right).
		\]
		After rescaling \( \zeta \) in $\mathbb{C}$, we obtain \( (dz, dw) \) as given in Table~\ref{table-two-singular-model}.
	\end{proof}
	
	Given an integral singular SKS on \( \Delta \) with singularity at 0, by Proposition~\ref{correspondence-singularity}, we can assume the integral structure is given by Table~\ref{table-two-singular-model}. Let \( f\colon X_{\Delta^*} \to \Delta^* \) be the torus fibration given by Proposition~\ref{bijection-regular-region}. By Kodaira's classical results on singular elliptic fibers \cite[Section 8]{kod23}, we can extend the fibration \( f\colon X_{\Delta^*} \to \Delta^* \) to \( f\colon X_{\Delta} \to \Delta \) in a canonical way, by adding a singular fiber over 0 according to the singular type in Table~\ref{table-two-singular-model}, and the zero-section over $\Delta^* $ naturally extends to $\Delta $.
	
	The canonical holomorphic 2-form \( \Omega \) on \( X_{\Delta^*} \) extends to a meromorphic form on \( X_{\Delta} \). By \cite[Table~1]{hein}, \( \Omega \) has no divisor along the singular fiber (note that our choice of generators \( \tau_1, \tau_2 \) is different from \cite[Table~1]{hein}, see also \cite[Section 5.1]{Chen2019CollapsingRM} for similar calculations). Thus, there is a canonically defined holomorphic volume form on \( X_{\Delta} \).

	\begin{proof}[Proof of Theorem~\ref{correspondence}]
		Let \( (B, \omega, J, \nabla) \) be an integral singular SKS with \( B = \mathbb{P}^1 \), and let \( S = \{ p_1, \dots, p_k \} \) be the singular locus, with \( B_0 = B \setminus S \) being the regular part. Let \( X_0 = T^* B_0 / \Lambda \) be the holomorphic torus fibration given by Proposition~\ref{bijection-regular-region}. From the discussion above, we can compactify \( X_0 \) to a complex surface \( X \) with trivial canonical bundle. By \cite[Proposition~2.1 in Chapter II]{friedman4mfd}, \( X \) is simply connected, so \( X \to B \) gives a Jacobian elliptic K3 surface, with the prescribed integral SKS.

		Conversely, given a Jacobian elliptic K3 surface \( f\colon X \to B \), then \( f \) has no multiple fibers since there is a global section. By Proposition~\ref{bijection-regular-region}, there is an integral SKS on the regular part \( B_0 \). By Kodaira's classification of singular fibers \cite{kod23} and calculation on the holomorphic form $\Omega$ \cite[Table~1]{hein}, the local models of the SKS given by the Jacobian elliptic K3 surface near the singular points are exactly those listed in Table~\ref{table-two-singular-model}. Thus, we can construct an integral singular SKS on the base.
	\end{proof}
	
	\begin{rmk}
		
		Here, we provide an alternative proof that a Jacobian elliptic K3 surface has at most 24 singular fibers, which does not rely on the Euler characteristic of K3 surfaces. By Theorem~\ref{correspondence}, a Jacobian elliptic K3 surface corresponds to an integral singular SKS on \( \mathbb{P}^1 \), and its singular fibers correspond to the singular points on \( \mathbb{P}^1 \). Let \( S \) be the set of singular points on \( \mathbb{P}^1 \). By \cite[Theorem~2 in Section 6.5]{kontsevich}, we have  
		\[
		\sum_{x\in S} i(x) = \chi(\mathbb{P}^1) = 2,
		\]
		where \( i(x) \in \frac{\mathbb{Z}}{12} \) depends on the local monodromy near \( x \). In our case, it can be verified that \( i(x) > 0 \) for all \( x \in S \). Thus, there are at most 24 singular points on \( \mathbb{P}^1 \).

	\end{rmk}

	Given a Jacobian elliptic K3 surface \( f\colon X \to \mathbb{P}^1 \), let \( (\mathbb{P}^1, g) \) denote the underlying metric space of the corresponding integral singular SKS given by Theorem~\ref{correspondence} .  By the construction of Chen et al. \cite[Theorem~1.1]{Chen2019CollapsingRM}, \( (\mathbb{P}^1, g) \) is the Gromov--Hausdorff limit of a sequence of hyperkähler metrics on K3 surfaces. Thus, by Theorem~\ref{correspondence}, we can define a map as in the Introduction:
	\[
	\mathcal{F}\colon \mathcal{M}_{\text{Jac}} \to \mathfrak{M}_2^c.
	\]
	Moreover, by Theorem~\ref{thm-integrality}, \( \mathcal{F} \) is surjective.
	
	\begin{rmk}
		Although we restrict to Jacobian elliptic K3 surfaces here, we remark that the construction of Chen et al. applies to all elliptic K3 fibrations.
	\end{rmk}

	\section{Properties of \( \mathcal{F} \)} \label{injectivity}
	
	We are interested in the following question: how many Jacobian elliptic K3 surfaces are associated with a given metric space \( (\mathbb{P}^1,g) \) by \( \mathcal{F} \)?
	
	Given a Jacobian elliptic K3 surface \( f\colon X \to \mathbb{P}^1 \), Theorem~\ref{correspondence} implies that this corresponds to an integral singular SKS on \( \mathbb{P}^1 \) after choosing a holomorphic volume form \( \Omega \) on \( X \). The image of \( (X, f) \) under \( \mathcal{F} \) is simply the underlying metric space of the integral singular SKS on \( \mathbb{P}^1 \) (rescaled to unit diameter).  Let \( (dz, dw) \) be an integral special pair associated with it. 
	
	If another Jacobian elliptic K3 surface \( (X', f') \) induces the same metric on \( \mathbb{P}^1 \), then by Proposition~\ref{sks-with-same-metric}, we can choose \( \Omega' \) on \( X' \) such that \( (X', f', \Omega') \) defines the same singular SKS on \( \mathbb{P}^1 \), possibly with a different integral structure given by \( (dz', dw') \).  
	
	If \( f\colon X \to \mathbb{P}^1 \) is isotrivial, then \( \frac{dw}{dz} \) is constant, and the metric on the regular part of \( \mathbb{P}^1 \) is flat. Conversely, if the metric on the regular part of \( \mathbb{P}^1 \) is flat, then by Proposition~\ref{sks-with-same-metric}, we have \( \frac{dw}{dz} \) constant, implying that \( f\colon X \to \mathbb{P}^1 \) is isotrivial. Thus, we divide the discussion into two cases.

	\subsection{Isotrivial Case}\label{section-isotrivial}
	
	In this case, the only possible singularity of SKS is type (III) with \( f = 0 \) in Definition~\ref{definition-singularity}. Near each singularity, the metric is locally a flat cone metric:
	\[
	\frac{\mathrm{i}}{2} |\zeta|^{2\beta - 2} d\zeta \wedge d\bar{\zeta},
	\]
	and we can take a local special pair 
	\[
	(dz_0, \mathrm{i} dz_0) = (\zeta^{\beta - 1} d\zeta, \mathrm{i} \zeta^{\beta - 1} d\zeta)
	\]
	with local monodromy \( R_{2\pi\beta} \). We divide the discussion into two cases.
	
	\
	
	\textbf{Case 1:} \( \beta \neq \frac{1}{2} \) for some $p_k$. Assume \( (dz', dw') \) and \( (dz, dw) \) are two integral special pairs. After some transformation in \( \mathrm{SL}_2(\mathbb{Z}) \), we can assume they give the same local monodromy at \( p_k \). Write \( (dz, dw) = (dz_0, \mathrm{i} dz_0) Q \) and \( (dz', dw') = (dz_0, \mathrm{i} dz_0) P Q \) for \( P, Q \in \mathrm{SL}_2(\mathbb{R}) \). Then we have
	\[
	Q^{-1} R_{2\pi\beta} Q = (PQ)^{-1} R_{2\pi\beta} (PQ) = Q^{-1} P^{-1} R_{2\pi\beta} P Q.
	\]
	By Lemma~\ref{commute-elliptic}, we have \( P = R_\theta \). Thus,
	\[
	(dz', dw') = (dz_0, \mathrm{i} dz_0) R_\theta Q = (e^{\mathrm{i} \theta} dz_0, e^{\mathrm{i} \theta} \mathrm{i} dz_0) Q = e^{\mathrm{i} \theta}(dz, dw),
	\]
	which is given by a twist by \( \theta \) of \( (dz, dw) \). 
	
	By Theorem~\ref{correspondence}, the two integral SKS give the same Jacobian elliptic K3 surfaces. Hence, \( \mathcal{F} \) is injective in this case.
	
	\
	
	\textbf{Case 2:}  \( \beta = \frac{1}{2} \) for all $p_k$. Then the local monodromy near every singularity is \( -\text{Id} \), and the monodromy is integral for any special pair. Furthermore, in this case, \( (\mathbb{P}^1, g) \) is a flat orbifold with local model \( \mathbb{R}^2 / \mathbb{Z}_2 \), thus it must be \( T^2 / \mathbb{Z}_2 \) with flat orbifold metric (see \cite[Chapter 13]{orbifold} for basic property of orbifolds), and is given by Example~\ref{example-kummer}. 
	The corresponding Jacobian elliptic K3 surface is given by $f\colon \widetilde{T^4/\mathbb{Z}_2}\to T^2 / \mathbb{Z}_2$, where $\widetilde{T^4/\mathbb{Z}_2}$ is the minimal resolution of $T^4/\mathbb{Z}_2$, thus a Kummer surface.

	In this case, the Jacobian elliptic K3 surface is determined by the \( j \)-invariant, $j(\frac{dw}{dz})$, and thus the preimage is parameterized by \( \mathbb{C} \).
	
	Combine Case 1 and Case 2 above, we proved the isotrivial part of Theorem~\ref{injective-theorem}.
	
	\subsection{Non-isotrivial Case}
	
	In this case, a twist by $\theta\notin \pi\mathbb{Z}$ must change the SKS by Proposition~\ref{twist-give-same-sks}. Thus, by Theorem~\ref{correspondence}, two integral special pairs \( (dz_1, dw_1) \), \( (dz_2, dw_2) \) of a given singular SKS on $\mathbb{P}^1$ give the same Jacobian elliptic K3 surface if and only if \( (dz_1, dw_1) = (dz_2, dw_2) P \) for some \( P \in \mathrm{SL}_2(\mathbb{Z}) \).
	
	\begin{defi}
		Given a monodromy representation \( \rho\colon \pi_1(M, p) \to \mathrm{SL}_2(\mathbb{Z}) \), we define the monodromy group \( G_\rho \) to be the image of \( \rho \) in \( \mathrm{SL}_2(\mathbb{Z}) \), and the modular monodromy group \( \bar{G}_\rho \) to be the image of \( G_\rho \) in \( \mathrm{PSL}_2(\mathbb{Z}) \).
	\end{defi}
	
	\begin{defi}
		Given a subgroup \( G \leq \mathrm{SL}_2(\mathbb{Z}) \), define
		\[
		N(G) = \# \left\{ P \in \mathrm{SL}_2(\mathbb{R}) / \mathrm{SL}_2(\mathbb{Z}) : P^{-1} G P \subset \mathrm{SL}_2(\mathbb{Z}) \right\}.
		\]
		By slight abuse of language, this also defines \( N(G) \) for \( G \leq \mathrm{PSL}_2(\mathbb{Z}) \).
	\end{defi}
	
	Take an integral singular SKS on \( \mathbb{P}^1 \) with monodromy group \( G_\rho \), assuming the metric is not flat on the regular part. By the discussion above, the number of its preimages under \( \mathcal{F} \) is given by \( N(G_\rho) \).
	
	In the non-isotrivial case, all possible monodromy groups \( G_\rho \) of elliptic K3 surfaces have finite index, with 3411 possible cases of \( G_\rho \) and 3228 possible cases of \( \bar{G}_\rho \) \cite{monodromygroup}. By \cite[Theorem~3.8 in Chapter II]{friedman4mfd}, the monodromy group is \( \mathrm{SL}_2(\mathbb{Z}) \) if the Jacobian elliptic K3 surface has 24 \( I_1 \)-fibers.

	As we will see in the next section, we have \( N(\mathrm{SL}_2(\mathbb{Z})) = 1 \), and \( N(G) \) is finite whenever \( G \) has finite index in \( \mathrm{SL}_2(\mathbb{Z}) \). Thus, we have proved the non-isotrivial case of Theorem~\ref{injective-theorem}.
	\begin{example}
		
		Consider the boundary of a symmetric 3-simplex, \( \partial \Delta \), which is homeomorphic to \( \mathbb{P}^1 \). Hultgren et al. \cite{HJMM} constructed a metric on \( \partial \Delta \) by solving the "tropical Monge-Ampère equation," and Jonsson et al. \cite{polytope} showed that this metric is given by an integral singular SKS on \( \partial \Delta \). Their method first fixes the affine structure and then uses optimal transport theory to produce the metric. This construction corresponds to a Jacobian elliptic K3 surface with \(6 I_4 \) singular fibers. In a suitable special pair, the monodromy along the six singularities is calculated to be:        
		\[
		\begin{pmatrix}
			1 & 4 \\
			0 & 1
		\end{pmatrix}, \quad
		\begin{pmatrix}
			5 & 4 \\
			-4 & -3
		\end{pmatrix}, \quad
		\begin{pmatrix}
			1 & 0 \\
			-4 & 1
		\end{pmatrix}, \quad
		\begin{pmatrix}
			-7 & 4 \\
			-16 & 9
		\end{pmatrix}, \quad
		\begin{pmatrix}
			-3 & 4 \\
			-4 & 5
		\end{pmatrix}, \quad
		\begin{pmatrix}
			-7 & 16 \\
			-4 & 9
		\end{pmatrix},
		\]
		and the monodromy group is calculated to be \( \Gamma(4) \) (Definition~\ref{definition-congruence-group}).
		
		After a transformation by \( \begin{pmatrix}\sqrt{2}  & 0 \\ 0 &  \frac{1}{\sqrt{2}}\end{pmatrix} \), the singular structure corresponds to \( 4I_2 + 2 I_8 \) singular fibers, the monodromy group becomes \( \Gamma_0(8) \cap \Gamma(2) \).
		
		After a transformation by \( \begin{pmatrix} 2 & 0 \\ 0 & \frac{1}{2} \end{pmatrix} \), the singular structure corresponds to \( 4I_1 + I_4 + I_{16} \) singular fibers, the monodromy group becomes \( \Gamma_0(16) \).
		
	\end{example}
	
	\subsection{Calculation of \( N(G) \)}\label{calculation-of-ng}
	
	Denote \( M_2 \) as the \( \mathbb{Z} \)-module of all integral 2-matrices. Let \( G \leq \mathrm{SL}_2(\mathbb{Z}) \) be a finite index subgroup, and let \( M(G) \) be the submodule of \( M_2 \) generated by \( G \). It is easy to see that \( P^{-1} G P \in \mathrm{SL}_2(\mathbb{Z}) \) if and only if \( P^{-1} M(G) P \in \mathrm{SL}_2(\mathbb{Z}) \).
	
	Write \( P = \begin{pmatrix} a_1 & a_2 \\ a_3 & a_4 \end{pmatrix} \in \mathrm{SL}_2(\mathbb{R}) \). Notice that
	\[
	\begin{aligned}
		& P^{-1} 
		\begin{pmatrix}
			x & y \\
			z & w
		\end{pmatrix}
		P = \begin{pmatrix}
			a_4 & -a_2 \\
			-a_3 & a_1
		\end{pmatrix}
		\begin{pmatrix}
			x & y \\
			z & w
		\end{pmatrix}
		\begin{pmatrix}
			a_1 & a_2 \\
			a_3 & a_4
		\end{pmatrix} \\
		=& x \begin{pmatrix}
			a_1a_4 & a_2a_4 \\
			-a_1a_3 & -a_2a_3
		\end{pmatrix}
		+ y \begin{pmatrix}
			a_3a_4 & a_4^2 \\
			-a_3^2 & -a_3a_4
		\end{pmatrix}
		+ z \begin{pmatrix}
			-a_1a_2 & -a_2^2 \\
			a_1^2 & a_1a_2
		\end{pmatrix}
		+ w \begin{pmatrix}
			-a_2a_3 & -a_2a_4 \\
			a_1a_3 & a_1a_4
		\end{pmatrix}.
	\end{aligned}
	\]

	If \( G = \mathrm{SL}_2(\mathbb{Z}) \), then \( M(G) = M_2 \), and thus \( a_i^2, a_i a_j \) are integral. Therefore, \( a_i = b_i \sqrt{t} \) for some integers \( b_i, t \), with \( t \) square-free. By \( a_1a_4 - a_2a_3 = 1 \), we have \( t = 1 \), and hence \( a_i \) are integral. Thus, \( N(\mathrm{SL}_2(\mathbb{Z})) = 1 \).
	
	\begin{lem}\label{nM2subset}
		Given a finite index subgroup \( G \leq \mathrm{SL}_2(\mathbb{Z}) \), there exists an \( n \in \mathbb{Z}^+ \) such that \( nM_2 \subset M(G) \).
	\end{lem}
	
	\begin{proof}
		Since \( G \) is a finite index subgroup of \( \mathrm{SL}_2(\mathbb{Z}) \), there exists \( n_i \in \mathbb{Z}^+ \) such that the following matrices are in \( G \):
		\[
		\begin{pmatrix} 1 & n_1 \\ 0 & 1 \end{pmatrix} = \begin{pmatrix} 1 & 1 \\ 0 & 1 \end{pmatrix}^{n_1}, \quad
		\begin{pmatrix} 1 & 0 \\ n_2 & 1 \end{pmatrix} = \begin{pmatrix} 1 & 0 \\ 1 & 1 \end{pmatrix}^{n_2}, \quad
		\begin{pmatrix} 1+n_3 & n_3 \\ -n_3 & 1-n_3 \end{pmatrix} = \begin{pmatrix} 2 & 1 \\ -1 & 0 \end{pmatrix}^{n_3}.
		\]
		
		From these three elements and the identity matrix \( \text{Id} \in G \), we can deduce that \( nM_2 \subset M(G) \) for some \( n \).
	\end{proof}
	
	Take \( n \) as in Lemma~\ref{nM2subset}, we have \( n a_i a_j \in \mathbb{Z} \), and we conclude that \( \sqrt{n} a_i = b_i \sqrt{t} \) for some integers \( b_i, t \) with \( t \) square-free. By \( a_1a_4 - a_2a_3 = 1 \), we have \( b_1b_4 - b_2b_3 = \frac{n}{t} \). Thus, we have \( \sqrt{\frac{n}{t}} P \) is an integral matrix with determinant equal to \( \frac{n}{t} \).
	
	We say that two matrices \( P \) and \( Q \) are equivalent if \( P = QX \) for some \( X \in \mathrm{SL}_2(\mathbb{Z}) \). The finiteness of \( N(G) \) follows from the following Lemma:
	
	\begin{lem}
		Given \( m \in \mathbb{Z}^+ \), there are only finitely many equivalence classes of integral matrices with determinant \( m \).
	\end{lem}
	\begin{proof}
		Let \( B = \begin{pmatrix} b_1 & b_2 \\ b_3 & b_4 \end{pmatrix} \) with determinant $m$. Take \( x = \gcd(b_1, b_2) \), and define \( c_1 = b_1 / x \), \( c_2 = b_2 / x \). Then, take \( c_3 \) and \( c_4 \) such that \( c_2 c_3 - c_1 c_4 = 1 \). We have
		\[
		\begin{pmatrix} b_1 & b_2 \\ b_3 & b_4 \end{pmatrix}\begin{pmatrix} c_4 & -c_2 \\ c_3 & -c_1 \end{pmatrix}= \begin{pmatrix} x & 0 \\ * & \frac{m}{x} \end{pmatrix}.
		\]
		Right-multiply by another \( \begin{pmatrix} 1 & 0 \\ k & 1 \end{pmatrix} \), we deduce that there exists a \( C \in \mathrm{SL}_2(\mathbb{Z}) \) such that \( BC \) is bounded. Thus, there are only finitely many equivalence class of such \( B \).
	\end{proof}

	In this way, we can calculate \( N(G) \) for any finite index subgroup of \( \mathrm{PSL}_2(\mathbb{Z}) \). We write \( G_1 \sim G_2 \) if \( G_2 = P^{-1} G_1 P \) for some \( P \in \mathrm{SL}_2(\mathbb{R}) \). By definition we have \( N(G_1) = N(G_2) \) if \( G_1 \sim G_2 \).
	
	We recall the definition of some classical subgroups of  \( \mathrm{PSL}_2(\mathbb{Z}) \) as follows:
	
	\begin{defi} \label{definition-congruence-group}  
		\[
		\bar{\Gamma}(m) \coloneq \left\{ A \in \mathrm{SL}_2(\mathbb{Z}), \quad A \equiv \begin{pmatrix} 1 & 0 \\ 0 & 1 \end{pmatrix} \pmod{m} \right\}/\{\pm I\},
		\]
		\[
		\bar{\Gamma}_1(m) \coloneq \left\{ A \in \mathrm{SL}_2(\mathbb{Z}), \quad A \equiv \begin{pmatrix} 1 & * \\ 0 & 1 \end{pmatrix} \pmod{m} \right\}/\{\pm I\},
		\]
		\[
		\bar{\Gamma}_0(m) \coloneq \left\{ A \in \mathrm{SL}_2(\mathbb{Z}), \quad A \equiv \begin{pmatrix} * & * \\ 0 & * \end{pmatrix} \pmod{m} \right\}/\{\pm I\},
		\]
		\[ \bar{\Gamma}_1(8;4,1,2) \coloneq\left\{ \begin{pmatrix} 1 + 4a & 2b \\ 4c & 1 + 4d \end{pmatrix}\in \mathrm{SL}_2(\mathbb{Z}) \mid a \equiv c \pmod{2} \right\}/\{\pm I\}, \]
		\[ \bar{\Gamma}_1(16;16,2,2) \coloneq \left\{ \begin{pmatrix} 1 + 4a & b \\ 8c & 1 + 4d \end{pmatrix}\in \mathrm{SL}_2(\mathbb{Z})  \mid a \equiv c \pmod{2} \right\}/\{\pm I\}, \]
		and $\bar{\Gamma}^n$ is defined to be the subgroup generated by $\{A^n:A\in\mathrm{PSL}_2(\mathbb{Z}) \}$ .

	\end{defi}
	
	In Table~\ref{table:index-group}, we calculate \( N(G) \) for certain modular monodromy groups of elliptic K3 surfaces. All torsion-free congruence groups that serve as modular monodromy groups of elliptic K3 surfaces are included (\cite[Corollary~ 5.6]{monodromygroup}; \cite[Theorem~5.1]{seb-classification-monodromy-group}).

	\begin{table}[h]
		\centering
		\caption{Calculation of \( N(G) \)}
		\renewcommand{\arraystretch}{1.3}
		\setlength{\tabcolsep}{10pt} 
		\begin{tabular}{|c|l|c|}
			\hline
			\textbf{Index} & \multicolumn{1}{c|}{\textbf{Group} \( G \leq \mathrm{PSL}_2(\mathbb{Z}) \)} & \textbf{\( N(G) \)} \\
			\hline
			\textbf{1} & \( \mathrm{PSL}_2(\mathbb{Z}) \) & 1 \\
			\hline
			\textbf{2} & \( \bar{\Gamma}^2 \) & 1 \\
			\hline
			\multirow{2}{*}{\textbf{3}} & \( \bar{\Gamma}^3 \) & 1 \\
			& \( \bar{\Gamma}_0(2) \) & 2 \\
			\hline
			\textbf{4} & \( \bar{\Gamma}_0(3) \) & 2 \\
			\hline
			\textbf{6} & \( \bar{\Gamma}(2) \sim \bar{\Gamma}_0(4) \) & 4 \\
			\hline
			\multirow{4}{*}{\textbf{12}} 
			& \( \bar{\Gamma}_1(5) \) & 2 \\
			& \( \bar{\Gamma}_0(6) \) & 4 \\
			& \( \bar{\Gamma}(3) \sim \bar{\Gamma}_0(9) \) & 5 \\
			& \( \bar{\Gamma}_0(8) \sim (\bar{\Gamma}_0(4) \cap \bar{\Gamma}(2)) \) & 6 \\
			\hline
			\multirow{5}{*}{\textbf{24}} 
			& \( \bar{\Gamma}_1(7) \) & 2 \\
			& \( \bar{\Gamma}_1(8) \) & 6 \\
			& \( \bar{\Gamma}_1(8;4,1,2) \sim \bar{\Gamma}_1(16;16,2,2) \) & 6 \\
			& \( \bar{\Gamma}_0(12) \sim (\bar{\Gamma}_0(3) \cap \bar{\Gamma}(2))  \) & 8 \\
			& \( \bar{\Gamma}(4) \sim \bar{\Gamma}_0(16) \sim (\bar{\Gamma}_0(8) \cap \bar{\Gamma}(2)) \) & 10 \\
			\hline
		\end{tabular}
		\label{table:index-group}
	\end{table}

	\bibliography{main}

\end{document}